\documentclass{article}
\usepackage[utf8]{inputenc}
\usepackage[english]{babel}
\usepackage{geometry}
\usepackage{fullpage} 
\usepackage[
]{hyperref}
\hypersetup{
	colorlinks,
	linkcolor={black},
	citecolor={blue!50!black},
	urlcolor={blue!80!black}
}
\usepackage{amssymb,bbm}
\usepackage{mathdots,mathtools}
\usepackage{xcolor,graphicx}
\usepackage[all]{xy}
\usepackage{stmaryrd}
\usepackage{wasysym}
\usepackage{mathbbol} 
\numberwithin{equation}{section}
\allowdisplaybreaks
\usepackage{subcaption}
\usepackage{easyReview}
\usepackage{amsthm}
\newtheorem{theorem}{Theorem}[section]

\newtheorem{proposition}[theorem]{Proposition}
\newtheorem{lemma}[theorem]{Lemma}
\newtheorem{corollary}[theorem]{Corollary}

\theoremstyle{remark}
\newtheorem{remark}{Remark}[section]
\theoremstyle{definition}

\usepackage{authblk}
\begin{document}

\title{The standard augmented multiplicative coalescent revisited}
\author[1]{Josué Corujo}
\author[2]{Vlada Limic}

\affil[1]{
Univ Paris Est Créteil, Univ Gustave Eiffel, CNRS, LAMA UMR 8050, F-94010 Créteil, France
}
\affil[1,2]{Institut de Recherche Mathématique Avancée, UMR 7501 Université de Strasbourg et CNRS, 
7 rue René-Descartes, 67000 Strasbourg, France}

{
	\makeatletter
	\renewcommand\AB@affilsepx{: \protect\Affilfont}
	\makeatother
	
	\affil[ ]{Email ids}
	
	\makeatletter
	\renewcommand\AB@affilsepx{, \protect\Affilfont}
	\makeatother
	
	\affil[1]{\href{mailto:josue.corujo-rodriguez@u-pec.fr}{josue.corujo-rodriguez@u-pec.fr}}
	\affil[2]{\href{mailto:vlada@math.unistra.fr}{vlada@math.unistra.fr}}
}

\maketitle
\begin{abstract}
The Erdős-Rényi random graph is the fundamental random graph model.
In this paper we consider its continuous-time version,
where multi-edges and self-loops are also allowed.
It is well-known that the sizes of its connected components evolve according to the multiplicative coalescent dynamics.
Moreover, with the additional information on the number of surplus edges, the resulting process follows the \emph{augmented multiplicative coalescent dynamic}, constructed by Bhamidi, Budhiraja and Wang in 2014.
The same authors exhibit the scaling limit, which can
be specified in terms of the infinite vector of excursions (in particular their lengths, and the areas enclosed by the excursion curves) above past infima of a reflected Brownian motion with linear infinitesimal drift.
We use some recent results, using a graph exploration process called the simultaneous breadth-first walk, to study the 
same scaling limit, called the \emph{standard augmented multiplicative coalescent} (SAMC). 
We present a self-contained, simpler and more direct approach than that of any previous construction of the SAMC.
Furthermore, we believe that the method described here is convenient for generalizations, one of which would be the study of general non-standard eternal augmented multiplicative coalescents.
\end{abstract}




\section{Introduction}

For $n \in \mathbb{N}$ we write $[n]$ to mean $\{1,\ldots, n\}$.
Let us denote by $G(n,p)$, where $p\in[0,1]$, the Erd\H{o}s-Rényi random graph: each edge is included in the graph with probability $p$, independently over edges.
A continuous-time variation of $G(n,p)$, where $p\in[0,1]$, is naturally constructed as follows: 
fix the set of vertices $[n]$ and let each of the $\binom{n }{2}$ edges appear at an exponential time of rate $1$, independently of each other.
This transforms the model into a homogeneous continuous-time Markov chain, running on the set of graphs with vertices $[n]$ and going from the trivial graph ($n$ disconnected vertices) at time $t = 0$, to the complete graph when $t \to \infty$.
Let us denote this process by $(G^{(n)}(t), t \ge 0)$.
This continuous-time construction is equally obtained by the time-change $t = -\ln(1-p)$ in the natural coupling of $\big(G(n,p), p \in [0,1]\big)$, i.e.\ the process where at each $p \in [0,1]$ the edge $\{i,j\}$ is present in the graph if and only if $\{ U_{\{i,j\}} \le p \}$, for a fixed family $ \displaystyle ( U_{\{i,j\}} )_{i\neq j \in [n]}$ of independent random variables with uniform distribution in $[0,1]$.

Here and in the rest of the paper \emph{connected} means connected by a path of edges in the usual graph theoretical sense. 
Let the \emph{mass} of any connected component be equal to the number of its vertices.
Due to elementary properties of independent exponential random variables, it is immediate that a pair of connected components merges at the rate equal to the product of their masses.
In other words, the vector of sizes of the connected components of the  continuous-time random graph evolves according to the \emph{multiplicative coalescent} (MC) dynamics:
\begin{equation}
\label{merge}
\begin{array}{c}
\mbox{ any pair of components with masses (sizes) $x$ and $y$ merges }\\
\mbox{ at rate $x \cdot y$ into a single component of mass $x+y$.}
\end{array}
\end{equation}
Due to the relation $p= 1-\mathrm{e}^{-t}$, this continuous-time random graph exhibits the same phase transition as $G(n,\cdot)$ does when $n$ diverges, at the critical time $1/n$ (plus a lower order term) where the giant component is born \cite{aldRGMC,bol-book,erdren}.
Let $\pmb{C}^{(n)}(t) = \big( |C_1^{(n)}(t)|, |C_2^{(n)}(t)|, \dots \big)$ be the vector such that $|C_i^{(n)}(t)|$ is the size of $C_i^{(n)}(t)$ which is the $i^{\mathrm{th}}$ largest component of $G^{(n)}(t)$.
In fact, taking $t_n := 1/n + t/n^{4/3}$ Aldous \cite{aldRGMC} characterized the
scaling limit of $n^{-2/3} \pmb{C}^{(n)}(t_n)$ when $n \to \infty$.
Indeed, let us define the metric space $(l^2_\searrow,d)$ of infinite sequences
$\pmb{y} = (y_1,y_2,\ldots)$
such that $y_1 \geq y_2 \geq \ldots \geq 0$
and $\sum_i y_i^2 < \infty$,
with the $l^2$-distance 
$\operatorname{d}(\pmb{y},\pmb{z} ) = \sqrt{\sum_i (y_i-z_i)^2}$.
For each $t\in \mathbb{R}$, let
\begin{equation}\label{eq:def_W}
W^{t}(s) = W(s) -  \frac{1}{2} s^2  +  t \, s,  \text{ for every } s \geq 0,
\end{equation}
where
$W$ is standard Brownian motion. 
Note that $W^t$ solves the stochastic differential equation
\(
    \mathrm{d} W^t_s = \mathrm{d} B_s + \mu_t(s) \mathrm{d}s,
\)
with a linear (infinitesimal) drift term $\mu_t(s) = t - s$.
Let
\begin{equation}\label{eq:def_B}
 B^{t}(s) := W^{t}(s) - \min_{0 \leq u \leq s} W^{t}(u), \text{ for every } s \geq 0.
\end{equation}
Aldous \cite{aldRGMC} proved that for each $t$, the infinite vector of ordered excursion lengths of $B^{t}$ away from $0$, denoted $\pmb{\mathcal{X}}(t)$, belongs to $l^2_\searrow$, and that
\begin{equation}\label{eq:convergenceMC}
n^{-2/3} \pmb{C}^{(n)}\left( \frac{1}{n} + \frac{t}{n^{4/3}} \right)
\text{ 
converges in distribution to } 
\pmb{\mathcal{X}}(t) \text{ in } (l^2_\searrow, \operatorname{d}).
\end{equation}
Furthermore, Aldous \cite{aldRGMC} also proved existence of an entrance law $(\pmb{\mathcal{X}}^\star(t), t \in \mathbb{R})$
such that for each $z>-\infty$, 
$(\pmb{\mathcal{X}}^\star(t), t \geq z)$
is a Feller process in $l^2_\searrow$ evolving according to \eqref{merge}, and such that the marginal $\pmb{\mathcal{X}}^\star(t)$ 
and $\pmb{\mathcal{X}}(t)$ have the same distribution for each $t\in \mathbb{R}$.
The process $(\pmb{\mathcal{X}}^\star(t), t \in \mathbb{R})$ is called the \emph{standard (eternal) multiplicative coalescent}.

\begin{remark}[Other related results]
Although it is not necessary to understand the rest of the article, we will now briefly discuss some other results related to multiplicative coalescents that could be convenient for the reader.
Aldous and Limic \cite{EBMC} extended \eqref{eq:convergenceMC} results to a general class of inhomogeneous random graphs and characterized all the marginal laws of the extreme eternal multiplicative coalescent.
More recently, Limic \cite{multcoalnew} characterized the trajectories of all the extreme version of the multiplicative coalescent.
In fact, defining $\pmb{\mathcal{X}}(t)$ as before, for every $t \in \mathbb{R}$, using the same Brownian motion we obtain a stochastic process $(\pmb{\mathcal{X}}(t), t \in \mathbb{R})$.
The results in \cite[Thm.\ 1.2]{multcoalnew} imply that this process evolves according to the multiplicative coalescent dynamic.
The statements of this kind, which characterize the trajectories of a process (in this case the multiplicative coalescent) could be called \emph{dynamic}.
In contrast to results like \eqref{eq:convergenceMC}, which characterize the marginal laws, and hence could be called \emph{static}.

For other research related to the Erd\H{o}s\,--\,Rényi random graph and other random graph models, with links to the eternal multiplicative coalescents we wish to refer the reader to \cite{Addario2012,bhamidietal2,Broutin2021,Broutin2022,bromar15,Dhara2017,Dhara2019} and the references therein.
\end{remark}

\subsection{The standard augmented multiplicative coalescent}

Let us now explain how to enrich the multiplicative coalescent dynamics, so that one can keep track of the non-spanning or the \emph{surplus} (also known as the \emph{excess}) edges in addition.
Recall that for a finite connected graph $(V, E)$ with $|V| = n$, the number of surplus edges is defined as $|E| - (n-1)$, since any spanning tree of $(V,E)$ must have exactly $n-1$ edges.
The number of surplus edges is a measure of how much a connected graph differs from a tree.

Let
\(
	\mathbb{N}^\infty = \{ (n_1, n_2, \dots): n_i \in \mathbb{N}_0, \text{ for all } i \ge 1 \},
\)
and define
\[
\mathbb{U}_\searrow = \left\{ (x_i, n_i)_{i \ge 1} \in l^2_\searrow \times \mathbb{N}^\infty : \sum_{i = 1}^\infty x_i \, n_i < \infty \text{ and } n_i = 0 \text{ whenever } x_i = 0, \, i \ge 1 \right\}\!.
\]
For every element $(x_i, n_i)_{i \ge 1} \in \mathbb{U}_\searrow$, the positive real $x_i$ represents the mass of the $i^{\mathrm{th}}$ largest component, while the positive integer $n_i$ is the number of surplus edges in the same component.
Let us define the following metric on $\mathbb{U}_\searrow$:
\begin{equation*}
 \mathbf{d}_{\mathbb{U}} \big( (\pmb{x}, \pmb{n}), (\pmb{x}', \pmb{n}') \big) = \left( \sum_{i = 1}^\infty (x_i - x_i')^2 \right)^{1/2} + \sum_{i = 1}^\infty |x_i \cdot n_i - x_i' \cdot n_i'|.
\end{equation*}
See Remark 4.15 in \cite{bhamidietal2} for a discussion on the choice of $\mathbf{d}_{\mathbb{U}}$ as the metric on $\mathbb{U}_\searrow$.
Consider also
\begin{equation} \label{def:Uzero}
\mathbb{U}_\searrow^0 = \left\{ (x_i, n_i)_{i \ge 1} \in \mathbb{U}_\searrow : \text{ if } x_k = x_m, \, \text{ for } k \le m, \text{ then } n_k \ge n_m \right\},
\end{equation}
which is convenient for working with discrete random graphs where different components may (and typically do) have equal total masses.

Counting the surplus edges in $({G}^{(n)}(t), t\ge 0)$ is not so convenient, due to the fact that one needs to keep track of the edges that already exist. 
This technical issue can be avoided by considering instead the following multi-graph version of the Erd\H{o}s\,--\,Rényi random graph, originally introduced by Bhamidi et al.\ \cite{bhamidietal2}.
For each $i,j\in [n]$ a new \emph{directed edge}
 $i \rightarrow j$ appears at rate $1/2$, and for each $i$ a self-loop $i \rightarrow i$ appears at rate $1/2$. 
This random-graph is in fact an oriented multi-graph, since it is an oriented graph with loops and multiple edges allowed.
Let us denote by $(MG^{(n)}(t), t\ge 0)$, this continuous-time multi-graph-valued Markov chain.
If the connected components are obtained by taking into account all the edges (regardless of their orientation), and the mass of each connected component is again the sum of masses of its participating vertices,
it is easy to see that the resulting ordered component masses evolve again according to the multiplicative coalescent transitions \eqref{merge}. 
Indeed, the presence of multi-edges and loops does not change the connectivity properties or the total component masses, so the random graph process from Bhamidi et al.~\cite{bhamidietal2} can be coupled to the Erd\H{o}s\,--\,Rényi random graph, and both are graph representations of the multiplicative coalescent.
A natural coupling could be realized as follows. The multi-graph $({MG}^{(n)}(t), t\ge 0)$ can be defined by associating a Poisson process of rate $1/2$ to each pair of vertices $i,j \in [n]$ accounting for the number of edges $i \to j$.
Thus, an edge (in either direction) linking $i$ and $j \neq i$ appears at rate $1$, while the self-loop $i \to i$ appears at rate $1/2$ for each $i \in [n]$.
Then, $({G}^{(n)}(t), t\ge 0)$ is simply obtained by only keeping track of the existence of at least one edge between each pair of vertices, and by forgetting the self-loops. 
The advantage of working with ${MG}^{(n)}$ rather than $G^{(n)}$ is that in ${MG}^{(n)}$ it is much easier to account for the number of surplus edges.

The joint evolution of component masses and surplus edge counts in ${MG}^{(n)}$ is a continuous-time Markov process in $(\mathbb{U}_\searrow^0, \mathbf{d}_{\mathbb{U}})$ (see \eqref{def:Uzero}) evolving as follows:
\begin{description}
 \item [coalescence jump] for each $i \neq j$,
 at rate $x_i \cdot x_j$,
the process jumps from $(\pmb{x}, \pmb{n})$ to 
 $(\pmb{x}^{i, j}, \pmb{n}^{i,j})$ 
 where $(\pmb{x}^{i, j}, \pmb{n}^{i,j})$  is obtained by merging components $i$ and $j$ into a component with mass $x_i + x_j$ and surplus $n_i + n_j$, followed by reordering the coordinates with respect to the masses in order to obtain again an element of $\mathbb{U}_\searrow^0$,
 \item [surplus jump] for each $i \ge 1$, at rate $x_i^2/2$,
 the process jumps from $(\pmb{x}, \pmb{n})$ to 
 $(\pmb{x}, \pmb{n}^{i})$, where $(\pmb{x}, \pmb{n}^{i})$ is the state obtained by changing only the $i^{\mathrm{th}}$ component $(x_i,n_i)$ of $(\pmb{x}, \pmb{n})$ into $(x_i,n_i+1)$,
and reordering the coordinates if needed, to obtain an element in $\mathbb{U}_\searrow^0$.
\end{description}
We will call a process with this dynamic an \emph{augmented multiplicative coalescent}, even in the setting where $\pmb{x}$ has a finite number of non-zero components.

Let us denote by $\Psi^{(n)}(t)$ the natural embedding in $\mathbb{U}^0_\searrow$ of the vector keeping track of the normalized sizes of the connected components in $MG^{(n)}(t)$, together with the number of (unlimited) surplus edges. 
More precisely, the scaling of connected component sizes is as in \eqref{eq:convergenceMC}, while the surplus edge counts are not rescaled,
so that 
$$
\Psi^{(n)}(t) = \left(\left(n^{-2/3}|C_i^{(n)}(t)|, \operatorname{SP}\big( C_i^{(n)}(t) \big) \right), \, i\geq 1\right),$$
where here and below $|C|$ and $\operatorname{SP}(C)$ denote the total mass and the number of surplus edges in the connected component $C$, respectively.
The main results of this paper concern the scaling limit of $\Psi^{(n)}(t_n)$, for $t_n = 1/n + t/n^{4/3}$.
Before stating them, let us introduce the limiting process.

Recall $W^t, B^t$ from \eqref{eq:def_W}--\eqref{eq:def_B}.
Let $\Lambda$ be a homogeneous Poisson point process on $[0,\infty)\times [0,\infty)$, independent of $W^t$ (and therefore of $B^t$).
Following  \cite{aldRGMC,bhamidietal2}, let $\Lambda^{t}(s)$ be the number of points in $\Lambda$
below the curve $u \mapsto B^{t}(u)$, $u\in [0,s]$. 
To each excursion of $B^{t}$ above $0$, one can assign a random ``mark count'' to be the increase in $\Lambda^{t}$  attained during this excursion (see \cite[Section 2.3.2]{bhamidietal2} for details). 
Let $\mathcal{N}_i(t)$ be this count corresponding to the $i^{\mathrm{th}}$ longest excursion of $B^{t}$, and $\pmb{\mathcal{N}}(t)  =(\mathcal{N}_1(t),\mathcal{N}_2(t),\ldots)$, so that for all $t \ge 0$
\begin{equation}
    \label{def:curlyZ}
    \pmb{\mathcal{Z}}(t) := (\pmb{\mathcal{X}}(t), \pmb{\mathcal{N}}(t)) \in \mathbb{U}_\searrow^0.
\end{equation}
We now state our main result.
\begin{theorem}\label{thm:main_result}
For every $t \in \mathbb{R}$,
\(  
\Psi^{(n)}\left(\frac{1}{n} + \frac{t}{n^{4/3}}\right)
\)
converges in distribution to $\pmb{\mathcal{Z}}(t)$ in $(\mathbb{U}^0_\searrow, \operatorname{d}_{\mathbb{U}})$ as $n\to \infty$.
\end{theorem}

An analogous statement is true in the limited surplus edge setting, or equivalently for the Erd\H{o}s\,--\,Rényi random (regular, not multi-) graphs.
Denote by $\hat{\varPsi}^{(n)}(q)$ the analogue of ${\Psi}^{(n)}(q)$, where $\hat{\varPsi}^{(n)}(q)$ keeps track of the number of surplus edges in each component of ${G}^{(n)}(q)$.

\begin{corollary}\label{corol:erdos-renyi_AMC}
For every $t \in \mathbb{R}$,
\(
\hat{\varPsi}^{(n)}\left(\frac{1}{n} + \frac{t}{n^{4/3}}\right)
\)
converges in distribution to $\pmb{\mathcal{Z}}(t)$ in $(\mathbb{U}^0_\searrow, \operatorname{d}_{\mathbb{U}})$ as $n\to \infty$.
\end{corollary}

To the best of our knowledge, statements equal or similar to Theorem \ref{thm:main_result} were previously shown at least three times.
Bhamidi et al.~originally constructed the augmented multiplicative coalescent
using a lengthy and complicated argument for Theorem \ref{thm:main_result} (see \cite[Thm.\ 5.1]{bhamidietal2}).
Broutin and Marckert used another exploration process, based on a Prim's order among the vertices, and relied on some results from \cite{Addario2012} which were not outlined in detail, proved a stronger (and dynamic) result (see \cite[Thm.\ 5]{bromar15}): 
\begin{equation*}
(\pmb{\mathcal{Z}}(t), t \in \mathbb{R}) \text{ is a version of the augmented multiplicative coalescent.}
\end{equation*}
Dhara et al.\ \cite[Thm.\ 3.3]{Dhara2017} proved a result analogous to Theorem \ref{thm:main_result} for the configuration model with finite third moment degree. Their argument is based on rather intricate (configuration) model-based analysis.
In their Remark 8.5 (rooted to their Proposition 8.4), Dhara et al.\ claim that
the scaling limits of different functionals, including the re-scaled component-sizes and the number of
surplus edges for the configuration model and the Erd\H{o}s\,--\,Rényi random graph are the same.
Corollary \ref{corol:erdos-renyi_AMC} confirms an analogous link between the multi-graphs of Bhamidi et al.~and the Erd\H{o}s\,--\,Rényi random graphs.

The interest of our novel approach is twofold: on the one hand, our arguments are self-contained and rather simple. 
On the other hand, our methods are based on the simultaneous breadth-first walk introduced by Limic \cite{multcoalnew} and recalled in Section \ref{sec:sBFW} below (for the particular setting of the Erd\H{o}s\,--\,Rényi random graph).
This L\'evy-type process remarkably encodes the trajectories of inhomogeneous random graphs
in a dynamic way, allowing us also to keep track of the number of surplus edges.
This relation is studied in the companion paper \cite{Corujo_Limic_2023a}.

The methods used here could potentially be generalized in the
general (augmented) multiplicative coalescent setting.
A work in progress started by the authors will cover this direction.

\section{Proof of Theorem \ref{thm:main_result} }

To study the behavior of the Erd\H{o}s\,--\,Rényi random graph in the critical regime, meaning for times of order $t_n = 1/n + t/n^{4/3}$, one usually considers a modification of the random graph process where each of the $n$ vertices has mass all equal to $1/n^{2/3}$, and every edge appears at rate equal to the product $1/n^{4/3}$ of their masses. 
In the modified model, the emergence of the giant (order $1$) component thus happens on the time scale $q_n(t) = n^{4/3} \cdot t_n = n^{1/3} + t$, where $t\in \mathbb{R}$.

Let us denote by $$\Big( \mathcal{G}^{(n)}(q), q \ge 0\Big) \text{ and } \Big( \mathcal{MG}^{(n)}(q), q \ge 0\Big)$$ the versions of $\Big( {G}^{(n)}(q), q \ge 0\Big) \text{ and } \Big( {MG}^{(n)}(q), q \ge 0\Big)$, respectively, with the masses of their vertices scaled by $n^{-2/3}$, as described above.
Then, if $\pmb{\mathcal{C}}^{(n)}(q) \in l^2_\searrow$ denotes the ordered vector of sizes of the connected components in $\mathcal{G}^{(n)}(q)$ (also in $\mathcal{MG}^{(n)}(q)$), then $(\pmb{\mathcal{C}}^{(n)}(q), q\ge 0)$ is a multiplicative coalescent started from a vector 
\begin{equation}
\label{def:x_n}
\pmb{x}^{(n)} = \left( \frac{1}{n^{2/3}}, \frac{1}{n^{2/3}}, \dots, \frac{1}{n^{2/3}}, 0, 0, \ldots \right) \in l^2_{\searrow},
\end{equation} 
of exactly $n$ non-zero initial components (the infinitely many components of mass $0$ do not interact during the evolution and are typically ignored).
Similarly, if $\pmb{\mathcal{Z}}^{(n)}(q)$ denotes the vector recording the size and  the number of surplus edge in each connected component in $\mathcal{MG}^{(n)}(q)$, then 
$(\pmb{\mathcal{Z}}^{(n)}(q), q\ge 0)$ is an augmented multiplicative coalescent started from 
$(\pmb{x}^{(n)}, \pmb{0}) \in \mathbb{U}^0_\searrow$, where $\pmb{0}$ stands for the null vector in $\mathbb{N}^\infty$.
The statement of Theorem \ref{thm:main_result} is equivalent to the following one: for each $t \in \mathbb{R}$,
\begin{equation*}
\pmb{\mathcal{Z}}^{(n)} \left( n^{1/3} + t \right)
\text{ converges in distribution to } \pmb{\mathcal{Z}}(t) \text{ in } (\mathbb{U}^0_\searrow, \operatorname{d}_{\mathbb{U}}), \text{ as } n\to \infty.   
\end{equation*}

\subsection{The simultaneous breadth-first walks} \label{sec:sBFW}

Recall $\pmb{x}^{(n)}$ from \eqref{def:x_n}.
For each $i \in [n]$, let $\xi_i$ have exponential distribution with rate $1/n^{2/3}$, independently over $i \in [n]$. 
Given $\pmb{\xi} := (\xi_1, \dots, \xi_{n})$, define simultaneously for all $s \ge 0$ and $q>0$
\begin{align*}
Z^{n,q}(s) :=& \sum_{i=1}^{n} \frac{1}{n^{2/3}} \mathbb{1}_{(\xi_i/q\, \leq \ s)} -s = \sum_{i=1}^{n} \frac{1}{n^{2/3}} \mathbb{1}_{(\xi_{(i)}/q\, \leq \ s)} -s, 
\end{align*}
where $(\xi_{(i)})_i$ are the ordered $(\xi_{i})_i$.
In words, $Z^{n,q}$ has a unit negative drift and successive positive jumps, which occur precisely at times $(\xi_{(i)}/q)_{i\leq n}$, and where the $i^{\mathrm{th}}$  successive jump is of magnitude ${1}/{n^{2/3}}$.
We call $\big( (Z^{n,q}(s), s \ge 0), q \ge 0\big)$ the \emph{simultaneous breadth-first walks (SBFW)}
(see \cite[\S\ 2]{multcoalnew} for more details).

Fix some $t \in \mathbb{R}$, let $q_n(t) = n^{1/3} + t$ and define
\begin{align*}
	Z^{(n)} \equiv Z^{(n),t} &: s \mapsto q_n(t) \cdot {Z}^{n, q_n(t)}(s),\\
	  B^{(n)} \equiv B^{(n),t} &: s \mapsto {Z}^{(n)}(s) - \inf\limits_{0 \le u \le s} {Z}^{(n)}(u).
\end{align*}

Limic \cite{multcoalnew}, using an approach rooted in \cite{aldRGMC,EBMC},
proved that the SBFW
encode the multiplicative coalescent process, as $q$ varies from $0$ to $\infty$.
The marginal (or static) result for a fixed $q$ or $q_n(t)$ is much simpler to derive, and was already known to Aldous \cite{aldRGMC}. 
In particular, the ordered vector of the excursion lengths of $B^{(n),t}$ above zero is an element of $l^2_\searrow$ (it has finitely many non-zero components), and it follows the same law as $ \pmb{\mathcal{C}}^{(n)}(q_n(t))$.

For a given $\pmb{x} \in l^2_\searrow$, let us define
\[
	\sigma_r(\pmb{x}) := \sum_i x_i^r, \text{ for } r = 1,2,3.
\]
In our setting $\sigma_1^{(n)} := \sigma_1(\pmb{x}^{(n)}) = n^{1/3}$, $\sigma_2^{(n)} := \sigma_2(\pmb{x}^{(n)}) = 1/n^{1/3}$ and $\sigma_3^{(n)} := \sigma_3(\pmb{x}^{(n)}) = 1/n$.
So, we have:
\[
{\sigma_{3}^{(n)}}/{\left(\sigma_{2}^{(n)}\right)^{3}} =1, \;\;
\sigma_{2}^{(n)} \rightarrow 0 \text{ and } {\| \pmb{x}^{(n)} \|_{\infty}}/{\sigma_{2}^{(n)}} \rightarrow 0,
\]
as $n \to \infty$.
For such a sequence $\big( \pmb{x}^{(n)} \big)_{n \ge 1}$ it is well-known (see \cite{aldRGMC,EBMC,multcoalnew}) that
for any $t \in \mathbb{R}$
\begin{equation*}
    \frac{1}{\sigma_2(\pmb{x}^{(n)})} \pmb{Z}^{n, q_n(t)}  \xrightarrow[n \to \infty]{d} W^{t}, 
\end{equation*}
where the convergence is with respect to the usual $J_1$-Skorohod topology on the space of càdlàg functions on $[0,\infty)$.
Figure \ref{fig:plot} shows a simulation of $Z^{(n)}$ and $B^{(n)}$, for $t = 1$ and $n = 1000$.

As a direct consequence, we also have
\begin{equation}
\label{eq:scaling_limit2}
    Z^{(n)} = \left( \frac{1}{\sigma_2(\pmb{x}^{(n)})} + t \right) \pmb{Z}^{n, q_n(t)} \xrightarrow[n \to \infty]{d} W^{t}, \text{ for each } t \in \mathbb{R},
\end{equation}
and
\begin{equation}
\label{eq:scaling_limit2B}
   B^{(n)}  \xrightarrow[n \to \infty]{d} B^{t}, \text{ for each } t \ge 0,
\end{equation}
where for \eqref{eq:scaling_limit2B} we use the fact that the past infimum of $Z^{(n)}$ is continuous and thus converges uniformly on compacts.
Moreover, using the Skorokhod representation theorem we can and will assume in the rest of the paper,
without a loss of generality, that $Z^{(n)}$, $B^{(n)}$, $W^t$ and $B^t$ are all defined on the same probability space, and that the convergence in \eqref{eq:scaling_limit2}--\eqref{eq:scaling_limit2} holds almost surely.

\begin{figure}
	\centering
	\includegraphics[width=0.9\linewidth]{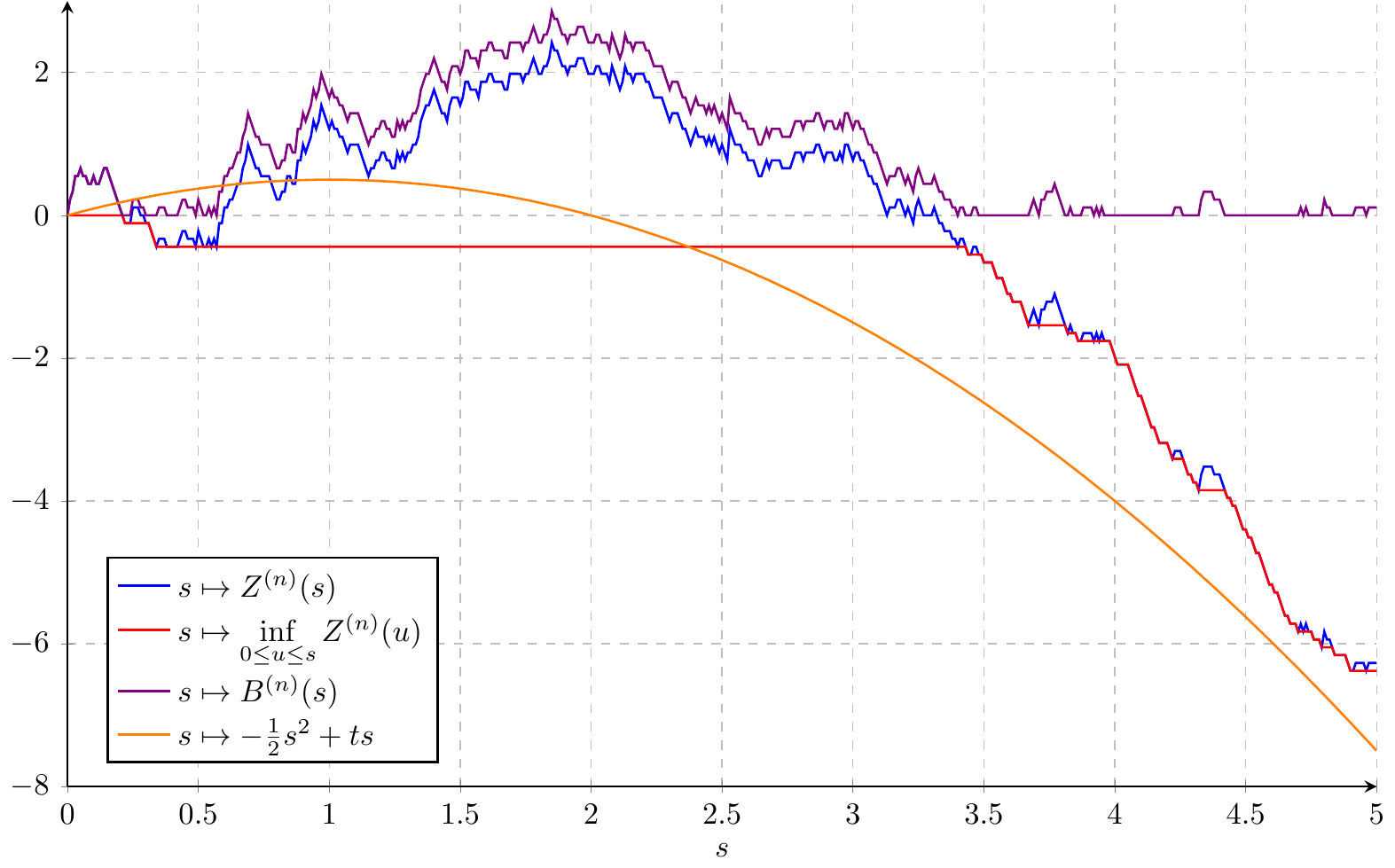}
	\caption{Plots of the curves $Z^{(n)}$, $\inf\limits_{0 \leq u \leq \cdot} Z^{(n)}$, $B^{(n)}$ and the parabola $s \mapsto - \frac{1}{2} s^2 + ts$, for $t = 1$ and with $n = 1000$.
	$Z^{(n)}$ converges to $B^t$, defined by \eqref{eq:def_W}, in the $J_1$ Skorohod topology.}
	\label{fig:plot}
\end{figure}

\subsubsection{Superimposing the surplus edges}

Let us denote by $S_i^{(n)}$ the area enclosed by the curve $s \mapsto {B}^{(n)}(s)$ during its $i^{\mathrm{th}}$ longest (we use largest interchangeably) excursion above zero.
Note that  $S_i^{(n)}$ is also equal to the area enclosed by the $i^{\mathrm{th}}$ largest excursion of ${Z}^{(n)}$ above its past infimum.
Let $a_i^{(n)}$ (resp.~$b_i^{(n)}$) be the lower (resp.~upper) limit of the $i^{\mathrm{th}}$ largest excursion interval.
Hence, the points $a_i^{(n)},b_i^{(n)}$ satisfy $B^{(n)}(a_i^{(n)})= B^{(n)}(b_i^{(n)})=0$ and 
$B^{(n)}(s)>0$, on $(a_i^{(n)},b_i^{(n)})$.
Let
 $X_i^{(n)}= b_i^{(n)} - a_i^{(n)} $ be the \emph{length} of the $i^{\mathrm{th}}$ longest excursion.
Proposition 5 in \cite{multcoalnew} implies that the vector $\pmb{X}^{(n)} = (X_1^{(n)}, X_2^{(n)}, \dots) \in l^2_\searrow$ has the law of the MC at time $q_n(t)$ started from $\pmb{x}^{(n)}$, i.e.\ the same law as $\pmb{\mathcal{C}}^{(n)}(q_n(t))$, the vector of component sizes of $\mathcal{G}^{(n)}(q_n(t))$ or $\mathcal{MG}^{(n)}(q_n(t))$.
Moreover, if we denote by $N_i^{(n)}$ the number of (unlimited) surplus edges in the (corresponding) $i^{\mathrm{th}}$ largest connected component, 
then given $B^{(n)}$ (or $Z^{(n)}$), $N_i^{(n)}$ has Poisson distribution with parameter (mean) 
\[
    {S}_i^{(n)} = \int_{a_i^{(n)}}^{b_i^{(n)}} {B}^{(n)}(s) \, \mathrm{d}s= 
    \int_{a_i^{(n)}}^{b_i^{(n)}} \Big( {Z}^{(n)}(s)- \inf_{u\leq a_i^{(n)}}{Z}^{(n)}(u) \Big) \,\mathrm{d}s,
\]
independently over different $i$ (see Theorem 1.1 in \cite{Corujo_Limic_2023a}).

If we denote by $\mathcal{C}_i^{(n)}$ the $i^{\mathrm{th}}$ largest component of the corresponding (coupled) random graph,
recall that 
\(
    X_i^{(n)} = {|\mathcal{C}_i^{(n)}|}
\),
where $|\mathcal{C}_i^{(n)}|$ is the total mass of $\mathcal{C}_i^{(n)}$.
It it often convenient (and compatible with the existing literature) to denote by $\operatorname{SP}(\mathcal{C}_i^{(n)})$ the number of surplus edges in $\mathcal{C}_i^{(n)}$.
We can rewrite the above facts as follows:
\[
    \operatorname{SP}(\mathcal{C}_i^{(n)}) \mid B^{(n)} \overset{d}{=} N_i^{(n)} \sim {\rm Poisson}\left( \int_{a_i^{(n)}}^{b_i^{(n)}} B^{(n)}(s) \,\mathrm{d}s \right),
\]
where the counters $(\operatorname{SP}(\mathcal{C}_i^{(n)}))_{i \ge 1}$ are conditionally independent given $B^{(n)}$ (or equivalently, given $Z^{(n)}$) over $i$.

We will write $\pmb{N}^{(n)}$ for the vector of surplus edge counts $(N_1^{(n)}, N_2^{(n)}, \dots)$.
The above observations lead to the conclusion that the natural embedding of $(\pmb{X}^{(n)}, \pmb{N}^{(n)})$ in $\mathbb{U}^0_\searrow$ has the same law as $\mathcal{Z}^{(n)}(q_n(t))$. 
Recall $\pmb{\mathcal{X}}$, $\pmb{\mathcal{N}}$ and $\pmb{\mathcal{Z}}$ from \eqref{def:curlyZ}.
Thus,
\begin{equation} \label{eq:metric_dU}    \operatorname{d}_\mathbb{U} \Big( (\pmb{X}^{(n)}, \pmb{N}^{(n)}) , \pmb{\mathcal{Z}}(t) \Big) = \left( \sum_{i \ge 1} \Big( X_i^{(n)} - \mathcal{X}_i(t) \Big)^2 \right)^{1/2} + \sum_{i \ge 1} \Big| X_i^{(n)} N_i^{(n)} - \mathcal{X}_i(t) \mathcal{N}_i(t) \Big|. 
\end{equation}
Due to Lemma 11 in \cite{multcoalnew} we know that
\begin{equation}
\label{eq:Xconvergence}
    \operatorname{d} ( \pmb{X}^{(n)} , \pmb{\mathcal{X}}(t) ) = \left( \sum_{i \ge 1} \Big( X_i^{(n)} - \mathcal{X}_i(t) \Big)^2 \right)^{1/2} \xrightarrow[n \to \infty]{\mathbb{P}} 0. 
\end{equation}
Due to the assumption about almost sure convergence that we made following \eqref{eq:scaling_limit2B},
let us note 
that on the same probability space we can assume that the convergence in \eqref{eq:Xconvergence}
also holds almost surely.
As a consequence, we have the almost sure convergence for each $i\geq 1$:
\begin{align*}
    X_i^{(n)} \xrightarrow[n \to \infty]{\mathrm{a.s.}} \mathcal{X}_i(t) 
    \text{ and } 
    N_i^{(n)} \xrightarrow[n \to \infty]{\mathrm{a.s.}} \mathcal{N}_i(t).
\end{align*}
So, in order to prove that $\operatorname{d}_\mathbb{U} ( (\pmb{X}^{(n)}, \pmb{N}^{(n)}) , \mathcal{Z}(t)) \to 0$ in probability, we will verify 
\begin{equation}
\label{eq:tightness}
    \lim_{N \to \infty} \limsup_{n} \mathbb{P}\left[ \sum_{i \ge N} X_i^{(n)} N_i^{(n)} > \epsilon \right] = 0,
\end{equation}
for each $\epsilon > 0$.
Moreover, large (long) excursions of $B^t$ and $B^{(n)}$ occur first, in the sense of Proposition 14 in \cite{EBMC}, so that when $N\uparrow \infty$ the sum in \eqref{eq:tightness} is only over short (small) excursions.
Hence, we only need to check that for every $\epsilon > 0$ 
\[
 \lim_{\delta \to 0} \limsup_{n}  \,\mathbb{P} \left[ \sum_{i: X_i^{(n)} < \delta} X_i^{(n)} \cdot N_i^{(n)} > \epsilon \right] = 0.
\]
Using the Markov inequality, we obtain
\[
    \mathbb{P} \left[ \sum_{i: X_i^{(n)} < \delta} X_i^{(n)} \cdot N_i^{(n)} > \epsilon \right] \le
    \frac{1}{\epsilon} \, \mathbb{E} \left[ \sum_{i \ge 1} X_i^{(n)} \cdot N_i^{(n)} \cdot \mathbb{1}_{\{X_i^{(n)} < \delta\}} \right].
\]
Thus, \eqref{eq:tightness} and Theorem \ref{thm:main_result} will follow from
\begin{equation}\label{eq:tight_to_prove}
    \lim_{\delta \to 0} \limsup_{n} \mathbb{E} \left[ \sum_{i \ge 1} X_i^{(n)} \cdot N_i^{(n)} \cdot \mathbb{1}_{\{X_i^{(n)} < \delta\}} \right] = 0. 
\end{equation}

Let $\mathcal{C}(v)$ denote the connected component which contains the vertex $v$, and let $|\mathcal{C}(v)|$ be the total mass of $\mathcal{C}(v)$.
Let $V_n$ by a vertex in $[n]$ (of mass $n^{-2/3}$) picked uniformly at random, i.e.
    $$
        \mathbb{P} (V_n=j) = \frac{1}{n}, \ j \in [n],
    $$
and let $V_n$ be independent of $Z^{(n)}$ and the coupled random graph.
The next result, inspired by Dhara et al.~\cite[Lemma 5.18]{Dhara2017}, is the first key ingredient in the proof of \eqref{eq:tight_to_prove}.

\begin{lemma}\label{lemma:random_vertex}
    If $f$ is a real valued function on the discrete space of connected graphs with fewer than $n$ vertices, then
    \begin{equation*}
        \mathbb{E} \big[ f\big(\mathcal{C}(V_n) \big) \big] = \frac{1}{n^{1/3}} \mathbb{E} \left[ \sum_{i \ge 1} |\mathcal{C}_i^{(n)}| \cdot f( \mathcal{C}_i^{(n)} ) \right],
    \end{equation*}
    where the sum runs over the sequence of connected components of the random graph.
\end{lemma}

\begin{proof}
    Denote by $\mathcal{F}^{(n)}$ the $\sigma$-algebra generated by $Z^{(n)}$ and the coupled random graph.
    Denote by $\Theta$ a randomly chosen connected component in the random graph, using the size-biased distribution
    \[
    \mathbb{P} \Big[ \Theta = \mathcal{C}_i \mid \mathcal{F}^{(n)} \Big] = \frac{ |\mathcal{C}_i^{(n)}| } {\sigma_1^{(n)}} = \frac{ |\mathcal{C}_i^{(n)}| } {n^{1/3}}.
    \]
Then
\begin{align*}
    \mathbb{E} \Big[ f(\Theta) \Big] &= \mathbb{E} \Big[ \mathbb{E} \big[ f(\Theta) \mid \mathcal{F}^{(n)} \big] \Big] = 
    \sum_{i \ge 1} \mathbb{E} \Big[ \mathbb{E} \big[ f(\Theta) \mid \mathcal{F}^{(n)} \big] \mathbb{1}_{\{ \Theta = C_i^{(n)} \}} \Big] = 
    \sum_{i \ge 1} \mathbb{E} \Big[ \mathbb{E} \big[ f(\Theta) \mathbb{1}_{\{ \Theta = C_i^{(n)} \}} \mid \mathcal{F}^{(n)} \big]  \Big] \\
    &= \sum_{i \ge 1} \mathbb{E} \Big[ \mathbb{E} \big[ f(C_i^{(n)}) \mathbb{1}_{\{ \Theta = C_i^{(n)} \}} \mid \mathcal{F}^{(n)} \big]  \Big] = \sum_{i \ge 1} \mathbb{E} \big[   \mathbb{P} \big[ \Theta = C_i^{(n)} \mid \mathcal{F}^{(n)} \big] f(C_i^{(n)}) \Big] \\
    &= \frac{1}{n^{1/3}} \mathbb{E} \Big[ \sum_{i \ge 1} |C_i^{(n)}| f(C_i^{(n)}) \Big].
\end{align*}
Clearly the sum in last identity could go over the connected components of the graph in any given order, and in particular in the ``chronological'' order encoded by $Z^{(n)}$ (or equivalently, by $B^{(n)}$).
    
The proof is concluded by showing that $\mathcal{C}(V_n) \overset{d}{=} \Theta$ given $\mathcal{F}^{(n)}$.
Indeed,
\begin{align*}
    \mathbb{P} \Big[ \mathcal{C}(V_n) = \mathcal{C}_i^{(n)} \mid \mathcal{F}^{(n)} \Big] &= \mathbb{P} \Big[ V_n \in \mathcal{C}_i^{(n)} \mid \mathcal{F}^{(n)} \Big] = \frac{1}{n^{1/3}} \sum_{j \in \mathcal{C}_j^{(n)}} \frac{1}{n^{2/3}} = \frac{|\mathcal{C}_i^{(n)}|}{n^{1/3}}. \qedhere
\end{align*}
\end{proof}

Recall that $|G|$ and $\operatorname{SP}(G)$ denote the total mass and the number of surplus edges, respectively, of a connected (sub)graph $G$.
Apply Lemma \ref{lemma:random_vertex} with $f(\mathcal{C}):= \operatorname{SP}(\mathcal{C})\mathbb{1}_{\{|\mathcal{C}| < \delta \}}$ to conclude
\begin{align}
   \mathbb{E} \left[ \sum_{i \ge 1} X_i^{(n)} \cdot N_i^{(n)} \cdot \mathbb{1}_{\{X_i^{(n)} < \delta\}} \right] 
   &= n^{1/3} \cdot \mathbb{E} \left[ \operatorname{SP}\big( \mathcal{C}(V_n) \big) \mathbb{1}_{\{|\mathcal{C}(V_n)| < \delta \}} \right]. \label{eq:with_sigma1}
\end{align}
Identity \eqref{eq:with_sigma1} is crucial in our approach.
We gradually reduced the study of the second summand (a tail sum) in the right-side of \eqref{eq:metric_dU}, to that of the component of a randomly chosen vertex, and the latter is a much more accessible quantity in our setting.

Another (simple but) key observation is that the (simultaneous) breadth-first walk picks the starting vertex of its exploration process in the size-biased order \cite[\S\ 2]{multcoalnew}.
Since all the vertices have equal mass, in our setting the size-biased pick equals to the uniform pick from $[n]$.
Thus, the first excursion above zero of $B^{(n)}$ naturally encodes $\mathcal{C}(V_n)$ in the sense that
\begin{align*}
    \Big( |\mathcal{C}(V_n)|, \operatorname{SP} \left( \mathcal{C}(V_n) \right) \Big) \overset{d}{=} 
    \left( T_1^{(n)}, \int_{\xi_{(1)}/q_n(t)}^{\xi_{(1)}/q_n(t) + T_1^{(n)}} B^{(n)}(s) \,\mathrm{d}s \right)
\end{align*}
where $T_1^{(n)} := \inf\{ s \ge 0: B^{(n)}(s + \xi_{(1)}/q_n) = 0 \}$ is the length of the first excursion of $B^{(n)}$ above zero, i.e.\ the one started at $\xi_{(1)}/q_n(t)$.
Therefore, we have 
\begin{equation}\label{eq:first_excursion}
\mathbb{E} \left[ \operatorname{SP}\big( \mathcal{C}(V_n) \big) \mathbb{1}_{\{|\mathcal{C}(V_n)| < \delta \}} \right]= {n^{1/3}}\cdot \mathbb{E} \left[ \int_{\xi_{(1)}/q_n(t)}^{\xi_{(1)}/q_n(t) + T_1^{(n)}} B^{(n)}(s) \,\mathrm{d}s \cdot  \mathbb{1}_{\{T_1^{(n)} < \delta \}}\right].
\end{equation}
So verifying \eqref{eq:tight_to_prove} reduces to proving the following lemma.
\begin{lemma}\label{lemma:key_lemma}
    \[
          \lim_{\delta \to 0} \limsup_{n} n^{1/3} \cdot \mathbb{E} \left[ \int_{\xi_{(1)}/q_n(t)}^{\xi_{(1)}/q_n(t) + T_1^{(n)}} B^{(n)}(s) \,\mathrm{d}s \cdot  \mathbb{1}_{\{T_1^{(n)} < \delta \}}\right] = 0
    \]
\end{lemma}
Even if the 
integral of $B^{(n)}$ over the first excursion is expected to be negligible,  we do not know a priori the rate at which it vanishes.
In the next section we prove that the quantities in \eqref{eq:first_excursion} are of order at most $\delta$ as $n \to \infty$.

\subsubsection{Proof of Lemma \ref{lemma:key_lemma}}

For simplicity and without any loss of generality, we will consider the case where $t = 0$, and thus $q_n := q_n(0) = 1/\sigma_2^{(n)} = n^{1/3}$.
Let $\widetilde{Z}^{(n)}$ be an auxiliary process defined as
\[
    \widetilde{Z}^{(n)}(s) = n^{1/3}\left(\sum_{i = 1}^n \frac{1}{n^{2/3}} N_i (s)  - s\right), \text{ for every } s \ge 0,
\]
where $(N_i(s), s\geq 0)$, $i\in [n]$
are $n$ independent Poisson processes with rate $1/n^{1/3} = q_n(0) / n^{2/3}$.
We can couple $Z^{(n)}$ and $\widetilde{Z}^{(n)}$, by declaring that $\xi_i/q_n(0)$ is the time of the first jump of $N_i(s)$, and then clearly $Z^{(n)}(s) \le \widetilde{Z}^{(n)}(s)$, for every $s \ge 0$, a.s.
The jumps of $\widetilde{Z}^{(n)}$ are of size $1/n^{1/3}$, and they arrive at rate $n^{2/3}$.
Indeed, if $T_j$ is the $j^{\mathrm{th}}$ jump time of $\widetilde{Z}^{(n)}$, then $T_{j+1}-T_j$ equals the minimum of $n$ independent residual exponential clock, where each of these $n$ competing clocks rings at rate $1/n^{1/3}$.
Consider the scaled process
\begin{align*}
    \widehat{Z}^{(n)}(u) &:= n^{1/3} \widetilde{Z}^{(n)} \left( \frac{u}{n^{2/3}} \right)  
    = S_u - u,
\end{align*}
where $(S_u, u \ge 0)$ is a Poisson process of rate $1$.
The excursions of $\widehat{Z}^{(n)}$ above past infimum are $n^{1/3}$ times higher and $n^{2/3}$ times longer than the corresponding excursions of $\widetilde{Z}^{(n)}$.
Therefore, their respective areas are $n$ times larger than the areas of the corresponding excursions of $\widetilde{Z}^{(n)}$.

More precisely, let $\widetilde{B}^{(n)}$ (resp.\ $\widehat{B}^{(n)}$) be the process $\widetilde{Z}^{(n)}$ (resp.\ $\widehat{Z}^{(n)}$) reflected above its past infimum.
We are interested in the integral of $\widetilde{B}^{(n)}$ over its first excursion.
If this excursion happens on the interval $[a^{(n)}, b^{(n)}]$, then the first excursion of $\widehat{B}^{(n)}$ happens on $[n^{2/3} a^{(n)}, n^{2/3} b^{(n)}]$ and $\widehat{B}^{(n)}$ is $n^{1/3}$ times larger than $\widetilde{B}^{(n)}$ on this interval.
Hence,
\begin{align*}
    \int\limits\limits_{n^{2/3} a^{(n)} }^{n^{2/3} b^{(n)}} \widehat{B}^{(n)} (u) \mathrm{d} u = \int\limits_{n^{2/3} a^{(n)} }^{n^{2/3} b^{(n)}} n^{1/3} \widetilde{B}^{(n)} \left( \frac{u}{n^{2/3}} \right) \mathrm{d} u = n \int\limits_{ a^{(n)} }^{  b^{(n)}} \widetilde{B}^{(n)} (s) \mathrm{d} s,
\end{align*}
where the last identity comes from the change of variables $s = u/n^{2/3}$.
Hence,
\begin{equation*}
\int\limits_{a^{(n)}_n}^{b^{(n)}} \widetilde{B}^{(n)} (s) \mathrm{d} s \cdot \mathbb{1}_{\{ b^{(n)} - a^{(n)} < \delta\}} = 
    \frac{1}{n} \int\limits_{n^{2/3} a_n}^{ n^{2/3} b_n } \widehat{B}^{(n)} (u) \mathrm{d} u \cdot \mathbb{1}_{\{ b^{(n)} - a^{(n)} < \delta\}}.
\end{equation*}
Furthermore, due to the coupling we established between $Z^{(n)}$ and $\widetilde{Z}^{(n)}$ described at the beginning of this section, we have
\begin{equation}\label{eq:stochastic_order}   
0\leq \int_{\xi_{(1)}/q_n}^{\xi_{(1)}/q_n + T_1^{(n)}} B^{(n)}(s) \,\mathrm{d}s \cdot  \mathbb{1}_{\{T_1^{(n)} < \delta \}} \ \le \  \int\limits_{a^{(n)}}^{b^{(n)}} \widetilde{B}^{(n)} (s) \mathrm{d}s \cdot \mathbb{1}_{\{b^{(n)} - a^{(n)} < \delta \}}. 
\end{equation}

The next result gives us a bound for the expectation of RHS in \eqref{eq:stochastic_order}.
Let $(S_u - u)_{u \ge 0}$ be a compensated rate $1$ Poisson process and let $\underline{\xi}$ and $\overline{\xi}$ be the left and right limits of its first excursion interval.
\begin{proposition}[Upper bound area enclosed by excursion] \label{proposition:upper_bound_standard_new_version}
For each $T \ge 0$ we have
\[
\mathbb{E} \left[ \int\limits_{\underline{\xi}}^{\overline{\xi}} \big( S_u - u + \underline{\xi} \big) \mathrm{d}u \cdot \mathbb{1}_{\{ \overline{\xi} - \underline{\xi} \le T \}} \right] \le T.
\]
\end{proposition}

Using Proposition \ref{proposition:upper_bound_standard_new_version}, we now complete the proof of Lemma \ref{lemma:key_lemma} as follows.

\begin{proof}[Proof of Lemma \ref{lemma:key_lemma}]

We can bound
\[
     \mathbb{E} \left[ \int\limits_{n^{2/3} a^{(n)}}^{ n^{2/3} b^{(n)} } \widehat{B}^{(n)} (u) \mathrm{d} u \cdot \mathbb{1}_{\{b^{(n)} - a^{(n)}<\delta\}} \right] \le \delta n^{2/3},
\]
and using in addition \eqref{eq:first_excursion}--\eqref{eq:stochastic_order}
we get 
\begin{align*}
n^{1/3}\mathbb{E} \left[ \operatorname{SP}(\mathcal{C}(V_n)) \mathbb{1}_{\{|\mathcal{C}(V_n)|<\delta\}} \right] &\le n^{1/3} \mathbb{E} \left[ \int\limits_{a^{(n)}}^{b^{(n)} } \widetilde{B}^{(n)} (u) \mathrm{d}u \cdot 
 \mathbb{1}_{\{b^{(n)} - a^{(n)}<\delta\}} \right]   \\
&\le n^{1/3}\frac{1 }{n} \mathbb{E} \left[ \int\limits_{n^{2/3} a^{(n)}}^{ n^{2/3} b^{(n)} } \widehat{B}^{(n)} (u) \mathrm{d} u \cdot \mathbb{1}_{\{b^{(n)} - a^{(n)}<\delta\}} \right] \le \frac{1 }{n^{2/3}} \delta n^{2/3} = \delta,
\end{align*}
which proves Lemma \ref{lemma:key_lemma}, and in turn Theorem \ref{thm:main_result}.
\end{proof}

\begin{proof}[Proof of Proposition \ref{proposition:upper_bound_standard_new_version}]
    Both $\underline{\xi}$ and $\overline{\xi}$ are stopping times with respect to the natural filtration $\big(\mathcal{F}^S_u\big)_{u \ge 0}$ of $(S_u)_{u \ge 0}$. 
    Note that the process $(M_u)_{u \ge 0}$, where 
    \[
    M_u := S_{u + \underline{\xi}} - (u + \underline{\xi})  + \underline{\xi} = S_{u + \underline{\xi}} - u, \text{ for every } u\geq 0
    \]
    is a compensated rate $1$ Poisson process started from $1$, and independent of $\mathcal{F}^S_{\underline{\xi}}$.
    Thus,
\[
    \mathbb{E} \left[ \int\limits_{\underline{\xi}}^{\overline{\xi}} (S_u - u + \underline{\xi}) \mathrm{d}u \cdot \mathbb{1}_{\{ \overline{\xi} - \underline{\xi} \le T \}} \right] =
    \mathbb{E} \left[ \int\limits_{0}^{\Xi} M_u \mathrm{d}u \cdot \mathbb{1}_{\{ \Xi \le T \}} \right]\!,
\]
where $\Xi := \inf \{ u \ge 0: M_u = 0 \}$.
Note that $\Xi$ is finite almost surely, but it has infinite expectation.

Let us consider the stopped martingale $(M_u^{\Xi\wedge T})_{u \ge 0}$ where $M^{\Xi\wedge T}_u := M_{u \wedge \Xi \wedge T}$,
and note that the expectation in the previous display is also equal to 
\[
    \mathbb{E} \left[ \int\limits_0^{\Xi} M^{\Xi \wedge T}_u \mathrm{d}u \cdot \mathbb{1}_{\{ \Xi \le T \}} \right]
    =
    \mathbb{E} \left[ \int\limits_0^{\Xi \wedge T} M^{\Xi \wedge T}_u \mathrm{d}u \cdot \mathbb{1}_{\{ \Xi \le T \}} \right] 
    =
    \mathbb{E} \left[ \int\limits_0^T M^{\Xi \wedge T}_u \mathrm{d}u \cdot \mathbb{1}_{\{ \Xi \le T \}} \right] \!.
\]
The last equality is due to the definition of $\Xi$, in general it would be an inequality in the direction which would enable the same conclusion as given below.

The stopped martingale $M^{\Xi \wedge T}$ is non-negative, and it satisfies  $
\mathbb{E}[M^{\Xi \wedge T}_u] = 1, \ \forall u \ge 0$.
Combining the above observations we conclude that 
\[
    \mathbb{E} \left[ \int\limits_{\underline{\xi}}^{\overline{\xi}} (S_u - u + \underline{\xi}) \mathrm{d}u \cdot \mathbb{1}_{\{ \overline{\xi} - \underline{\xi} \le T \}} \right]
=    \mathbb{E} \left[ \int\limits_0^T M^{\Xi \wedge T}_u \mathrm{d}u \cdot \mathbb{1}_{\{ \Xi \le T \}} \right] 
\leq 
 \mathbb{E} \left[ \int\limits_0^T M^{\Xi \wedge T}_u \mathrm{d}u \right]
 = \int\limits_0^T \mathbb{E} \left[M^{\Xi \wedge T}_u\right] \mathrm{d}u  = T,
\]
where due to the non-negativity of  $M^{\Xi \wedge T}$ we have the middle inequality, and also can apply the Fubini\,--\,Tonelli theorem. 
\end{proof}

\section{Proof of Corollary \ref{corol:erdos-renyi_AMC}}

Consider the natural coupling between $\mathcal{MG}^{(n)}$ and $\mathcal{G}^{(n)}$: for any $i\neq j\in [n]$, only the first edge connecting $i$ and $j$ in $\mathcal{MG}^{(n)}$ is accounted for in $\mathcal{G}^{(n)}$.
As already pointed out, both process are graph representations of the multiplicative coalescent, and they have the same connected component structure at all times (while their surplus edges are counted differently).
Set $q_n(t) = n^{1/3} + t$, and let $N_i^{(n)}$ and $\hat{N}_i^{(n)}$ be the numbers of surplus edges in the $i^{\mathrm{th}}$ largest component of $\mathcal{MG}^{(n)}(q_n(t))$ and $\mathcal{G}^{(n)}(q_n(t))$, respectively.

Define $\hat{\pmb{N}}^{(n)} := (\hat{N}_1^{(n)}, \hat{N}_2^{(n)}, \dots)$ so that $(\pmb{X}^{(n)}, \hat{\pmb{N}}^{(n)}) \in \mathbb{U}^0_\searrow$.
Due to Theorem \ref{thm:main_result}, in order to prove Corollary \ref{corol:erdos-renyi_AMC}, we only need to check that
\begin{equation} \label{eq:convergence_ER} 
\operatorname{d}_{\mathbb{U}} \Big( (\pmb{X}^{(n)}, \pmb{N}^{(n)}), ({\pmb{X}}^{(n)}, \hat{\pmb{N}}^{(n)}) \Big) = \sum_{i \ge 1} X_i^{(n)} (N_i^{(n)} - \hat{N}_i^{(n)}) \xrightarrow[n \to \infty]{ \mathbb{P}} 0. 
\end{equation}
As before, we will separately analyze the initial (largest) terms, and the tails of the sum.
Note that
\begin{equation}
\label{eq:tails_ER}
    \sum_{i: X_i^{(n)} < \delta}  X_i^{(n)} (N_i^{(n)} - \hat{N}_i^{(n)})  \le \sum_{i: X_i^{(n)} < \delta}  X_i^{(n)} N_i^{(n)},
\end{equation}
so we can use the results of the previous section to bound the sum in \eqref{eq:tails_ER} and prove that it is negligible as $n\to \infty$. 
Now we focus on the convergence of individual terms in  \eqref{eq:convergence_ER}.
Note that for any $\epsilon > 0$ we have
\begin{equation}
    \label{eq:N_difference_est}
    \mathbb{P} \Big[ X_i^{(n)} (N_i^{(n)} - \hat{N}_i^{(n)}) > \epsilon \Big] \le  \mathbb{P} \Big[N_i^{(n)} - \hat{N}_i^{(n)} \ge 1 \Big] \le \mathbb{E} [N_i^{(n)} - \hat{N}_i^{(n)}].
\end{equation}
We therefore need to control, for any given $i \ge 1$, the difference between the expected numbers of surplus edges in the $i^{\mathrm{th}}$ largest component of $\mathcal{MG}^{(n)}(q_n(t))$, and the (same) $i^{\mathrm{th}}$ largest component of $\mathcal{G}^{(n)}(q_n(t))$.

For this we briefly recall the construction of the surplus edges on the top of the spanning tree coupled to the simultaneous breadth-first walk, from \cite{Corujo_Limic_2023a}.
Given $Z^{(n)}$, the $i^{\mathrm{th}}$ largest component of $\mathcal{MG}^{(n)}(q_n(t))$ is coupled to the $i^{\mathrm{th}}$ largest excursion of $B^{(n)}$ above zero, which occurs during the interval $[a_i^{(n)}, b_i^{(n)}]$ of length $X_i^{(n)}$.
Furthermore, for each $j \in C_i^{(n)}$, the number of self-loops $j\to j$ in $\mathcal{MG}^{(n)}(q_n(t))$ is a Poisson random variable $\zeta_{j,j}$ with expectation 
\(
    \Delta_{j,j} = \frac{q_n(t)}{2 \, n^{4/3}}.
\)
Similarly, for each pair of vertices $j,k \in C_i^{(n)}$, given
$Z^{(n)}$, the number of surplus edges between vertices $j$ and $k$ is a Poisson random variable $\zeta_{j,k}$ with expectation \[
\Delta_{j,k} = (q_n(t) - T_{j,k}) / n^{4/3},
\]
 where $T_{j,k}$ is the
 time at which $j$ and $k$ become connected (by a path in $\mathcal{MG}^{(n)}$).
In particular, we have
\begin{equation}
\label{eq:surplus_edge_rate_bound}
    \Delta_{j,k} \le \frac{q_n(t)}{n^{4/3}}.
\end{equation}
All these surplus counting Poisson random variables are mutually conditionally independent given $Z^{(n)}$.

In order to bound $\mathbb{E} [N_i^{(n)} - \hat{N}_i^{(n)}]$ we will now split the set of 
surplus 
edges into three disjoint subsets.
The first subset are the self-loops, which are only accounted for in  $\mathcal{MG}^{(n)}(q_n(t))$ and never in $\mathcal{G}^{(n)}(q_n(t))$. 
In expectation they amount to
\begin{align*}
    E_1^{(n)} := \mathbb{E} \left[ \sum\limits_{j \in C_i^{(n)}} \zeta_{j,j} \right] = \mathbb{E} \left[ \sum \limits_{j \ge 1} \zeta_{j,j} \cdot \mathbb{1}_{\{ j \in C_i^{(n)} \}}\right] = \frac{q_n(t)}{2 \, n^{4/3}} \mathbb{E}\left[ |C_i^{(n)}| \right] = \frac{q_n(t)}{2 \, n^{2/3}} \mathbb{E}\left[ X_i^{(n)} \right],
\end{align*}
where the second to last identity was obtained via nested conditioning on $Z^{(n)}$.

The second subset consists of the surplus edges in $\mathcal{MG}^{(n)}(q_n(t))$ which are superimposed on the spanning 
edges of $C_i^{(n)}$ (these surplus edges are impossible in $\mathcal{G}^{(n)}(q_n(t))$).
It is here not necessary to know which edges are the spanning ones, but we do know that there are $|C_i^{(n)}|-1$ many,  and that given $Z^{(n)}$, a surplus multi-edge over a spanning edge $\{j,k\}$ is created at a rate 
smaller than $q_n(t)/n^{4/3}$
(see \eqref{eq:surplus_edge_rate_bound}).
Let us denote by $\mathcal{E}^{\mathrm{span}}$ the set of  
spanning edges of 
$C_i^{(n)}$.
Then, again via nested conditioning 
\[
    E_2^{(n)} := \mathbb{E} \left[ \sum\limits_{(j,k) \in \mathcal{E}^{\mathrm{span}}} \zeta_{j,k} \right] = \mathbb{E} \left[ \sum\limits_{(j,k) \in \mathcal{E}^{\mathrm{span}}} \Delta_{j,k} \right] \le \frac{q_n(t)}{n^{4/3}} \mathbb{E} \left[ |C_i^{(n)}| \right] = \frac{q_n(t)}{n^{2/3}} \mathbb{E} \left[ X_i^{(n)} \right].
\]

The third subset consists of the non-spanning non-self-loops edges, for which only the first surplus edge is accounted for in $\mathcal{G}^{(n)}(q_n(t))$, while in $\mathcal{MG}^{(n)}(q_n(t))$ there are potentially additional multiple surplus edges superimposed.
Recall that if $\zeta$ is a
Poisson random variable with expectation $\Delta$, then
\(
    \mathbb{E}[\zeta - \zeta \wedge 1] = \Delta - (1 - \mathrm{e}^{-\Delta}) \le {\Delta^2}/{2}.
\)
Hence, again via nested conditioning and using \eqref{eq:surplus_edge_rate_bound} we get
\begin{align*}
    E_3^{(n)} := \mathbb{E} \left[ \sum\limits_{ \substack{j \neq k \in C_i^{(n)} \\ (j,k) \notin \mathcal{E}^{\mathrm{span}} } }    \zeta_{j,k} - \zeta_{j,k} \wedge 1 \right] &\le \mathbb{E} \left[ \sum\limits_{ \substack{j \neq k \in C_i^{(n)} \\ (j,k) \notin \mathcal{E}^{\mathrm{span}} } } \frac{(\Delta_{j,k})^2}{2} \right] \\
    &\le 
    \frac{1}{2} \left(\frac{q_n(t)}{ n^{4/3}} \right)^2 \mathbb{E} \left[ |C_i^{(n)}|^2 \right] = \frac{q_n(t)^2}{2 \, n^{4/3}} \mathbb{E} \left[ (X_i^{(n)})^2 \right].
\end{align*}
Notice that $\mathbb{E}[X_i^{(n)}] = \Theta(1)$ and $\mathbb{E}[(X_i^{(n)})^2] = \Theta(1)$. 
This can be seen from \eqref{eq:Xconvergence}, the Cauchy\,–-\,Schwarz inequality, and the finiteness of $\mathbb{E}[\mathcal{X}_i^2]$ for each $i$.
Combining the three estimates above we can now conclude 
\[
    \mathbb{E}\left[ N_i^{(n)} - \hat{N}_i^{(n)} \right] \le E_1^{(n)} + E_2^{(n)} + E_3^{(n)} \xrightarrow[n \to \infty]{} 0, 
\]
which together with \eqref{eq:N_difference_est} yields Corollary \ref{corol:erdos-renyi_AMC}.

\section*{Acknowledgments}

This work of the Interdisciplinary Thematic Institute IRMIA++, as part of the ITI 2021-2028 program of the University of Strasbourg, CNRS and Inserm, was supported by IdEx Unistra (ANR-10-IDEX-0002), and by SFRI-STRAT’US project (ANR-20-SFRI-0012) under the framework of the French Investments for the Future Program.

\providecommand{\bysame}{\leavevmode\hbox to3em{\hrulefill}\thinspace}
\providecommand{\MR}{\relax\ifhmode\unskip\space\fi MR }
\providecommand{\MRhref}[2]{%
	\href{http://www.ams.org/mathscinet-getitem?mr=#1}{#2}
}
\providecommand{\href}[2]{#2}

\end{document}